\documentclass{amsart}
\usepackage{amssymb}
\usepackage{amsfonts}
\usepackage{amssymb}
\usepackage{amsmath, accents}
\usepackage{amsthm}
\usepackage{enumerate}
\usepackage{tabularx}
\usepackage{centernot}
\usepackage{mathtools}
\usepackage{stmaryrd}
\usepackage{amsthm,amssymb}
\usepackage{etoolbox}
\usepackage{tikz}
\usepackage{tikz-cd} 
\usepackage{amssymb}
\usetikzlibrary{matrix}
\usepackage{marginnote}
\definecolor{mygray}{gray}{0.85}
\usepackage[backgroundcolor=mygray,colorinlistoftodos,prependcaption,textsize=small]{todonotes}
\usepackage{xcolor}

\renewcommand{\leq}{\leqslant}
\renewcommand{\geq}{\geqslant}

\makeatletter
\def\subsection{\@startsection{subsection}{3}%
  \z@{.5\linespacing\@plus.7\linespacing}{.3\linespacing}%
  {\bfseries\centering}}
\makeatother

\makeatletter
\def\subsubsection{\@startsection{subsubsection}{3}%
  \z@{.5\linespacing\@plus.7\linespacing}{.3\linespacing}%
  {\centering}}
\makeatother

\makeatletter
\def\myfnt{\ifx\protect\@typeset@protect\expandafter\footnote\else\expandafter\@gobble\fi}
\makeatother

\newtheorem{theorem}{Theorem}[section]

\newtheorem{corollary}[theorem]{Corollary}
\newtheorem{definition}[theorem]{Definition}
\newtheorem{lemma}[theorem]{Lemma}

\newtheorem{example}[theorem]{Example}

\newtheorem{fact}[theorem]{Fact}

\newtheorem{strategy}[theorem]{Strategy}
\newtheorem{remark}[theorem]{Remark}

\newtheorem{notation}[theorem]{Notation}

\newtheorem*{1.1theorem}{Theorem 1.1}
\newtheorem*{1.2theorem}{Theorem 1.2}
\newtheorem*{1.1.1claim}{\normalfont \emph{Claim 1.1.1}}
\newtheorem*{2.20.1claim}{\normalfont \emph{Claim 2.20.1}}
\newtheorem*{2.20.2claim}{\normalfont \emph{Claim 2.20.2}}
\newtheorem*{4.4.1claim}{\normalfont \emph{Claim 4.4.1}}

\newcommand{\proofclaim}{\emph{Proof of Claim}}

\usepackage{eso-pic}
\usepackage{fancyhdr}

\usepackage{pdfpages}
\usepackage{hyperref}
\usepackage{cleveref} 
\usepackage{lmodern}

\begin{document}
	
	\begin{abstract} In \cite{sklinos} Sklinos proved that any uncountable free group is not $\aleph_1$-homogeneous. This was later generalized by Belegradek in \cite{bele} to torsion-free residually finite relatively free groups, leaving open whether the assumption of residual finiteness was necessary. In this paper we use methods arising from the classical analysis of relatively free groups in infinitary logic to answer Belegradek's question in the negative. Our methods are general and they also apply to varieties with torsion, for example we show that if $V$ contains a finite non-nilpotent group, then any \mbox{uncountable $V$-free group is not $\aleph_1$-homogeneous.}
	\end{abstract}
	
	\title[Non homogeneity of relatively free groups]{The construction principle and non homogeneity of uncountable relatively free groups}
	
	
	\thanks{The second named author was supported by project PRIN 2022 ``Models, sets and classifications", prot. 2022TECZJA and by INdAM Project 2024 (Consolidator grant) ``Groups, Crystals and Classifications”. We would like to thank R. Sklinos for pointing out to us the work of Belegradek \cite{bele} which motivated and inspired the writing of the present paper.}
	
	\author{Davide Carolillo}
	
	\address{Department of Mathematics ``Giuseppe Peano'', University of Torino, Via Carlo Alberto 10, 10123, Italy.}
	\email{davide.carolillo@unito.it}
	
	\author{Gianluca Paolini}

	\address{Department of Mathematics ``Giuseppe Peano'', University of Torino, Via Carlo Alberto 10, 10123, Italy.}
	\email{gianluca.paolini@unito.it}

	\date{\today}
	\maketitle
	
	
	\newcommand{\mbf}{\mathbf}
	\newcommand{\msc}{\mathcal}
	\newcommand{\mbb}{\mathbb}
	\newcommand{\op}{\operatorname}
	\newcommand{\join}{\vee}
	\renewcommand{\Join}{\bigvee}
	\newcommand{\meet}{\wedge}
	\newcommand{\Meet}{\bigwedge}
	\newcommand{\la}{\langle}
	\newcommand{\ra}{\rangle}
	\newcommand{\ol}{\overline}
	\newcommand{\ul}{\underline}
	\newcommand{\qimp}{\rightarrow}
	\newcommand{\onto}{\twoheadrightarrow}
	\newcommand{\wavy}{\approx}
	\newcommand{\DD}{\,\ol{\ul{\op{D}}}\,}
	
	\newcommand{\tp}{\text{\normalfont tp}}
	\newcommand{\fvo}{F_V(\omega_1)}
	\newcommand{\fvl}{F_V(\lambda)}
	\newcommand{\fvk}{F_V(\kappa)}
	\newcommand{\varv}{V}
	\newcommand{\wordsv}{\mathbf{V}}
	\newcommand{\varw}{W}
	\newcommand{\wordsw}{\mathbf{W}}
	\newcommand{\varab}{\mathrm{Ab}}
	\newcommand{\varg}{\mathrm{Grp}}
	\newcommand{\cat}{\text{\normalfont Cat}}
	\newcommand{\inc}{I}
	\newcommand{\Hom}{\text{\normalfont Hom}}
	\newcommand{\catC}{\mathcal{C}}
	\newcommand{\catD}{\mathcal{D}}
	\newcommand{\p}{{\normalfont\textbf{P}}}
	\newcommand{\cp}{{\normalfont\textbf{CP}}}
	\newcommand{\cspace}{~}
	\newcommand{\say}[1]{``#1''}
	
	\section{Introduction}
	
	Given a cardinal $\kappa$ and a structure $M$ of power $\geq \kappa$, we say that $M$ is  $\kappa$-homogeneous (or $\kappa$-type-homogeneous) if whenever two tuples of elements from $M$ of length $<\kappa$ have the same first-order type in $M$, then they are automorphic. Recently, the question of homogeneity received considerable attention in the model theoretic community, in particular in the context of groups. A fundamental result in this direction is the work of Perin and Sklinos \cite{perin} and, independently, of Houcine \cite{houcine} showing that any non-abelian free group is $\aleph_0$-homogeneous, thus answering a question of Sela. Again Sklinos, in \cite{sklinos}, proved that any uncountable free group is {\em not} $\aleph_1$-homogeneous. In his proof of this latter fact Sklinos used methods from stability theory, in particular the theory of forking, thus relying on the fundamental result of Sela that the theory of non-abelian free groups is stable \cite{sela}. This work was later continued by Belegradek in \cite{bele}, where he showed that on one hand more elementary methods were already sufficient to obtain non $\aleph_1$-homogeneity, and on the other hand that the proof applied to relatively free groups in any torsion-free variety of groups whose free objects were residually finite. The point of the assumption of residual finiteness in Belegradek's work is the fact that, by a classical result of Maltsev, any finitely generated residually finite group is Hopfian, and this latter fact is a key ingredient of Belegradek's proof. On the other hand, as already observed by Belegradek in \cite{bele}, there exist  torsion-free varieties of groups in which all relatively free groups
	of rank more than one are non-Hopfian \cite{hopfian}. Consequently, Belegradek in \cite[pg.\cspace783]{bele} asks whether the assumption of residual finiteness is necessary to conclude non $\aleph_1$-homogeneity of an uncountable torsion-free relatively free group. This is the motivating question of the present paper, \mbox{which we solve in Theorem~\ref{main_cor}.}
	
	\medskip
	
	The core of our approach to the solution of Belegradek's question is the use of some technology developed in the context of infinitary logic in dealing with the classical problem of logical characterization of the free groups in a variety $V$, or, from a more algebraic point of view, the problem of existence of $V$-$\aleph_1$-free groups which are not $V$-free, where we recall that given a variety of groups $V$, we say that $G \in V$ is $V$-$\aleph_1$-free whenever every countable subgroup of $G$ is $V$-free. The literature on the subject is is extensive (see e.g. \cite{eklof,mekler2}). Some stepping stones are:
	\begin{enumerate}[(1)]
		\item there exists a free group which is $\aleph_1$-free but not free \cite{higman}\footnote{At the best of our knowledge,\cspace\cite{higman} is the first reference on the subject of almost-free groups.};
		\item there exists an abelian group which is $\aleph_1$-free but not free abelian \cite{mekler1};
		\item the generalization of the above to any torsion-free variety of groups $V$ \cite{eklof, pope1984}.
	\end{enumerate}
	
	\smallskip
	
	Behind all these constructions there is a combinatorial principle known as {\em the construction principle} $(\cp)$ (cf. Definition\cspace\ref{definition CP}), and in fact large part of the literature quoted above deals  with the analysis of this principle and with the determination of which varieties $V$ of algebras satisfy the construction principle when relativized to  $V$. The crucial point of this paper is the realization that this combinatorial principle is most relevant for the problem of homogeneity. We prove: 
	\begin{theorem}\label{theorem V satisfies (CP) -> Fv(k) not homoheneous} Let $V$ be a variety of groups. If $(\cp)$ holds in $\varv$, then for every uncountable $\kappa$, $F_V(\kappa)$ is not $\aleph_1$-homogeneous. Furthermore, $F_V(\aleph_0)$ has an elementary subgroup which is not a $V$-free factor of $F_V(\aleph_0)$.
	\end{theorem} 

	Theorem~\ref{theorem V satisfies (CP) -> Fv(k) not homoheneous} gives us a new technical tool to prove non $\aleph_1$-homogeneity of $V$-free groups which does not require residual finiteness, and which opens the way for a number of new applications. In this perspective, an analysis of \cite{mekler_groups} allows us to conclude the following, thus in particular answering Belegradek's question:
	
	\begin{theorem}\label{main_cor} Let $V$ be a variety of groups satisfying at least one the following:
		\begin{enumerate}[(1)]
			\item $V$ is torsion-free;
			\item $V$ is not nilpotent but locally nilpotent;
			\item $V$ contains a finite non-nilponent group.
		\end{enumerate}
		Then, for every $\kappa \geq \aleph_1$ we have that $F_\varv(\kappa)$ is not $\aleph_1$-homogeneous. Furthermore, $F_V(\aleph_0)$ has an elementary subgroup which is not a $V$-free factor of $F_V(\aleph_0)$. 
	\end{theorem}
	
	Notice in particular that Theorem~\ref{main_cor}\,\emph{(2)-(3)} deal also with varieties with torsion, and so our methods are much more general than what is needed to solve Belegradek's question. On the other hand, clearly not all varieties of groups $V$ are such that the uncountable $V$-free groups are not $\aleph_1$-homogeneous, for example, as well-known and also mentioned in \cite[Remark~2]{bele}, in any variety $\varv$ of abelian groups of exponent $n$ for $2 \leq n < \omega$ we have that the $V$-free groups are saturated and thus homogeneous.
	Finally, we believe that our methods extend naturally to any variety of algebras which satisfy the construction principle.
	
	\smallskip\noindent
	The authors wish to thank the referee for their valuable suggestions, which have significantly improved the final version of this paper.
\section{Preliminaries}
Let $L$ be a first-order language, $M$ be an $L$-structure and $\bar{a}$ be a tuple of elements of $\text{dom}(M)$, not necessarily finite. The \emph{type} of $\bar{a}$ in $M$, denoted by $\tp^M(\bar{a})$, is the set of all first-order $L$-formulas that $\bar{a}$ realizes in $M$. 
\begin{definition}
	Let $\kappa$ be a cardinal. An $L$-structure $M$ is said to be $\kappa$\emph{-type-homogeneous} $($or simply $\kappa$\emph{-homogeneous}$)$ if and only if, for any tuples $\bar{a}$ and $\bar{b}$ in $M$ s.t. $\tp^A(\bar{a})=\tp^A(\bar{b})$, and $|\bar{a}|,|\bar{b}|<\kappa$, there exists some $\alpha\in\text{\normalfont Aut}(M)$ which maps $\bar{a}$ into $\bar{b}$. If $\kappa=|M|$, we just say that $M$ is \emph{type-homogeneous} $($or \emph{homogeneous}$)$.
\end{definition}
\begin{remark}
	Observe that $M$ is $\kappa$-type-homogeneous iff, for every partial elementary map $\sigma$ on $M$ s.t. $\text{dom}(\sigma)<\kappa$, there exists some $\alpha\in\text{\normalfont Aut}(M)$ s.t. $\sigma\subseteq\alpha$. 
\end{remark}
For the rest of this paper, $L=\{\,\cdot\,,(\,\text{-}\,)^{-1},e\}$ will be the language of groups. Further details about homogeneous structures are available in \cite{marker}. A detailed introduction to group varieties is given in \cite{neumann}. A \emph{variety} of groups $V$ is a class of groups defined equationally, i.e., the class of models of the $L$-theory obtained by adding to the theory of groups the universal closures of some positive $L$-atomic formulas. Any variety $V$ uniquely yields a family of positive universal sentences $\forall \bar{x} \big(w(\bar{x})=e\big)$, which are satisfied by all groups in $\varv$, where $w(\bar{x})$ are reduced words\footnote{I.e. terms not containing expressions of the form $xx^{-1}$, or $x^{-1}x$, for some variable $x$ in $\bar{x}$.} of the language of groups, called \emph{laws} of $\varv$. We denote varieties with capital letters ($\varv,\varw,\ldots$), and the corresponding sets of laws with boldface letters ($\wordsv,\wordsw,\ldots$).
\begin{definition}\label{definition V-basis}
	Let $\varv$ be a variety of groups and let $A\in \varv$. A subset $X$ of $A$ is said to be $\varv$-\emph{independent} iff for any group word $w(\bar{z})$ and $\bar{x}\in X^{|\bar{z}|}$ the equality $w(\bar{x})=e$ holds in $A$ just in case $w(\bar{z}\,)\in\wordsv$. If $A\in\varv$ is generated by a $\varv$-independent subset $X$, we say that $X$ is a $\varv$-\emph{basis} of $A$ and that $A$ is $\varv$-\emph{free} $($or simply \emph{free}\footnote{Sometimes, these groups are referred to in the literature as \emph{relatively free}. Here, we prefer a more informative expression.} \emph{in }$\varv$$)$.
\end{definition}
\noindent
Of course, these concepts depend on $\varv$: for instance, $\mathbb{Z}^2$ is free in the variety of abelian groups $\varab$, but not in the variety of all groups $\varg$. According to \cite{neumann}, we will refer to groups which are free in $\varg$ as \emph{absolutely free} groups.
\begin{remark}\label{remark: universal property}
A group $A\in \varv$ is $\varv$-free with basis $X$ if and only if it satisfies the following diagram property, known as \emph{the Universal Property (for $\varv$):} for all $B\in \varv$, any function from $X$ to $B$ extends uniquely to \mbox{a homomorphism from $A$ to $B$.}
\end{remark}
\noindent

\smallskip\noindent
It is known that any two $\varv$-bases of a $\varv$-free group $A$ have the same cardinality (cf. 13.53 of \cite{neumann}). This cardinality is called \emph{the rank of }$A$. For any positive cardinal $\kappa$, there is a $\varv$-free group of rank $\kappa$, which is unique up to isomorphism. We denote this group as $F_\varv(\kappa)$, or simply $F(\kappa)$ for the absolutely free group of rank $\kappa$.

\smallskip\noindent
Every variety $\varv$ naturally yields a category $\text{Cat}(\varv)$, whose arrows are group homomorphisms, and whose domain is $\varv$ itself. In this sense, Remark \ref{remark: universal property} just states that a $\varv$-free group is a free object of the category $\text{Cat}(\varv)$.
\begin{remark}\label{remark bijection of V-bases yields an automorphism of a V-free algebra}
	A direct consequence of the Universal Property for $\varv$ is that any bijection between two $\varv$-bases $X$ and $Z$ of a $\varv$-free group $A$ uniquely yields an automorphism of $A$ identifying $X$ and $Z$.
\end{remark}
\smallskip\noindent
It is well-known (c.f.\cspace\cite{lyndon-schupp}, pg.\,174-175) that free products admit characterization in terms of a diagram property. Arguing as in \cite{neumann} (Def.\cspace18.11), we extend this notion to the case of a generic variety: this characterization will be useful for our purpose\footnote{Formally, Definition\cspace\ref{definition V-free product} is not the one which appears in \cite{neumann}, but it is equivalent to it.}.
\begin{definition}\label{definition V-free product}
	Let $\varv$ be a variety of groups. A $\varv$-group $A$ is said to be the $\varv$-\emph{free product} of its subgroups $B$ and $C$ if and only if it is generated by $B$ and $C$ and, for any group $D\in \varv$ and any homomorphisms $\beta:B\rightarrow D$, and $\gamma:C \rightarrow D$, there is a unique homomorphism $\alpha:A \rightarrow D$ s.t. $\alpha\restriction_B=\beta$ and $\alpha\restriction_C=\gamma$. If this is the case, we also say that $B$ is a $\varv$-\emph{free factor} of $A$, that $C$ is \mbox{a \emph{complementary factor} of $B$ in $A$, and we write $B \leq_* A$.}
\end{definition}
\noindent
In categorical terms, Definition\cspace\ref{definition V-free product} states that the $\varv$-free product is the coproduct of the category $\cat(\varv)$. Recurring examples of such products in literature are the free product, and the direct sum, which are the coproducts of the categories $\text{Cat}(\varg)$ and $\cat(\varab)$, respectively. Just as in these cases, Definition\cspace\ref{definition V-free product} naturally extends to an arbitrary number of $\varv$-free factors, although we will almost always consider only the product of two factors.

\smallskip\noindent
The property satisfied by $A$, $B$ and $C$ in Definition\cspace\ref{definition V-free product} is known in Category Theory as Universal Property. In the following arguments, it will often occur in combination with the Universal Property of Remark\cspace\ref{remark: universal property}. So, in order to avoid confusion, we will simply refer to it as \say{diagram property}. 

\smallskip\noindent
We are primarily interested in the study of $\varv$-free factors as they provide a natural criterion for distinguishing between subgroups that are well-embedded in an ambient $\varv$-free group, and subgroups which are not. In the first case we will be able to extend any automorphism of the subgroup to an automorphism of the whole ambient group, in the second case this will not be possible in general. We formalize this statement in the following lemmas.
\begin{notation}
	We denote by \say{$\leq$} the subgroup relation. If $B\leq A$ are groups, and $S$ is a subset of $A$, then we let:
	\begin{enumerate}[(1)]
		\item $\langle S\rangle_A$ be the subgroup of $A$ generated by $S$;
		\item $N_A(B)=\langle aba^{-1}:a\in A,\,b\in B\rangle_A$ be the normal closure of $B$ in $A$;
		\item $\text{\normalfont id}_A$ be the identity function on $A$.
	\end{enumerate}
\end{notation}
\noindent
First, we recall a direct consequence of the Diagram Property of $\varv$-free products, which corresponds in our setting to Property\cspace$(2)$ of \cite{bele}.
\begin{lemma}\label{lemma extension automorphisms of coproducts}
	Let $\varv$ be a variety of groups and let $A\in\varv$. If $B$ is a subgroup of $A$ s.t. $B\leq_* A$, then any automorphism of $B$ extends to an automorphism of $A$.
\end{lemma}
\begin{proof}
	Let $C$ be a $\varv$-free complementary factor of $B$ in $A$. By Definition \ref{definition V-free product}, every pair of automorphisms $\beta\in\text{Aut}(B)$ and $\gamma\in\text{Aut}(C)$ uniquely yields an endomorphism $\alpha:A\rightarrow A$ s.t. $\alpha\restriction_B={\beta_\star}$, and $\alpha\restriction_C={\gamma_\star}$, with ${\beta_\star}$ and ${\gamma_\star}$ denoting the compositions of $\beta$ and $\gamma$ with the natural inclusions of $B$ and $C$ in $A$, respectively.
		
		\smallskip\noindent
		To see that $\alpha$ is an automorphism, it suffices to repeat the previous argument on $\beta^{-1}\in\text{Aut}(B)$ and $\gamma^{-1}\in\text{Aut}(C)$. This gives a homomorphism $\alpha':A\rightarrow A$ s.t. $\alpha'\restriction_B=({\beta^{-1}})_\star$ and $\alpha'\restriction_C=(\gamma^{-1})_\star$, where, as before, $(\beta^{-1})_\star$ and $(\gamma^{-1})_\star$ denote the compositions of $\beta^{-1}$ and $\gamma^{-1}$ with the inclusions of $B$ and $C$ in $A$, respectively. By construction, the following equalities hold:
		\begin{equation*}
			(\alpha\restriction_B\circ\,\alpha'\restriction_B)=\text{id}_B,\quad(\alpha\restriction_C\circ\,\alpha'\restriction_C)=\text{id}_C.
		\end{equation*}
		\noindent
		Since these conditions are realized also by the identity on $A$, $\alpha'$ must be a right inverse of $\alpha$ (\,i.e. $\alpha\circ\alpha'=\text{id}_A$). By symmetry, we can show that it is a left inverse as well. So, $\alpha'=\alpha^{-1}$, and $\alpha$ is an automorphism of $A$, as desired.
\end{proof}
\noindent
A direct consequence of Definition\cspace\ref{definition V-free product} and the Universal Property of $\varv$-free groups is that the $\varv$-free products of $\varv$-free groups is still a $\varv$-free group with a $\varv$-basis given by the union of the $\varv$-bases of the factors. This observation allows us to give the following useful characterization of Definition \ref{definition V-free product} in terms of $\varv$-bases of $\varv$-free groups.
\begin{lemma}\label{lemma characterization of f.f.}
	Let $\varv$ be a variety of groups, and $A\in\varv$ be $\varv$-free. Then, for every $\varv$-free subgroup $B$ of $A$ the following are equivalent:
	\begin{enumerate}[(1)]
		\item $B\leq_*A$ and there is a $\varv$-free complementary factor of $B$ in $A$;
		\item there is a $\varv$-basis of $B$ extending to a $\varv$-basis of $A$;
		\item every $\varv$-basis of $B$ extends to a $\varv$-basis of $A$.
	\end{enumerate}
\end{lemma}
\begin{proof}
	Implication $[3\Rightarrow2]$ is immediate. For $[2\Rightarrow1]$, let $Y$ be a $\varv$-basis of $B$ and $X$ be a $\varv$-basis of $A$ s.t. $Y\subseteq X$. We claim that $B\leq_*A$ and $C\vcentcolon=\langle X\setminus Y\rangle_A$ is a complementary factor of $B$ in $A$. Since $X$ is a set of generators for $A$, we only need to show that the Diagram Property of Definition \ref{definition V-free product} is satisfied by the triple $A,B,C$.
	
	\smallskip\noindent
	To this extent, consider a $\varv$-group $D$, and some homomorphisms $\beta:B\rightarrow D$, and $\gamma:C\rightarrow D$. By the Universal Property of $A$, there exists a unique homomorphism $\alpha:A\rightarrow C$ s.t. $\beta\restriction_Y\cup \,\gamma\restriction_{X\setminus Y}\,\subseteq \alpha$. Clearly, $\alpha\restriction_B\,=\beta$, by the Universal Property of $B$, as $Y$ is a $\varv$-basis of $B$ and $\alpha\restriction_Y\,=\beta\restriction_Y$, by construction. Similarly, $\alpha\restriction_C\,=\gamma$, since $\alpha\restriction_{X\setminus Y}\,=\gamma\restriction_{X\setminus Y}$ and $X\setminus Y$ is a $\varv$-basis of $C$.
	
	\smallskip\noindent
	Finally, the uniqueness of $\alpha$ is ensured by the Universal Property of $A$. Indeed, if $\alpha':A\rightarrow D$ were a homomorphism s.t. $\alpha'\restriction_B\,=\beta$ and $\alpha'\restriction_C\,=\gamma$, then we would also have: $\beta\restriction_Y\,\cup\, \gamma\restriction_{X\setminus Y}\,\subseteq\alpha'$. So, $\alpha$ and $\alpha'$ would necessarily be equal, by the Universal Property of $A$, as desired.
	
	\smallskip\noindent
	Finally, for implication $[1\Rightarrow3]$, let $C$ be a $\varv$-free complementary factor of $B$ in $A$. Given any two $\varv$-bases $Y,Z$ of $B$ and $C$, respectively, we claim that $X\vcentcolon=Y\cup Z$ is a $\varv$-basis of $A$.
	
	\smallskip\noindent
	Clearly, $X$ generates $A$, as so does $B\cup C$. So, it only remains to prove that for any $D\in \varv$ and any mapping $f:X\rightarrow D$ there is a unique homomorphism from $A$ to $D$ extending $f$. Notice that $f$ yields the functions $f\restriction_Y\,:Y\rightarrow D$ and $f\restriction_Z\,:Z\rightarrow D$. Hence, by the Universal Properties of $B$ and $C$, there are two homomorphisms $\beta:B\rightarrow D$ and $\gamma:C\rightarrow D$ which are the unique satisfying conditions $f\restriction_Y\,\subseteq\beta$ and $f\restriction_Z\,\subseteq\gamma$. Now, the Diagram Property of Definition\cspace\ref{definition V-free product} ensures the existence of a unique homomorphism $\alpha:A\rightarrow D$ s.t. $\alpha\restriction_B\,=\beta$ and $\alpha\restriction_C\,=\gamma$.
	
	\smallskip\noindent
	In particular, $f=f\restriction_Z\cup\, f\restriction_Z\,\subseteq \alpha$. To see that it is the unique with this property, consider some homomorphism $\alpha':A\rightarrow D$ s.t. $f\subseteq\alpha$. Then,
	\begin{equation*}
		\alpha'\restriction_B=\alpha\restriction_B=\beta \quad\text{and}\quad \alpha'\restriction_B=\alpha\restriction_B=\beta,
	\end{equation*}
	by the Universal Properties of $B$ and $C$, respectively. However, by the Diagram Property of $B,C$ and $A$, this means that $\alpha'=\alpha$. Therefore, $X$ a $\varv$-basis of $A$ extending $Y$, as desired.
\end{proof}
\noindent
\begin{lemma}\label{lemma existence of V-free complementary factor}
	Let $\varv$ be a variety of groups, $F$ be a $\varv$-free group of infinite rank and $A$ be a $\varv$-free group s.t. $A\leq_*F$ and $\mathrm{rank}(A)\leq\mathrm{rank}(F)$. Suppose that there is a $V$-basis $X$ of $F$ admitting a subset $X_A$ s.t. $A\leq\langle X_A\rangle_F$ and $|X\setminus X_A|=|X|=\mathrm{rank}(F)$. Then, $A$ has a complementary factor in $F$ which is also a $\varv$-free group.
\end{lemma}
\begin{proof}
	Since $X$ is infinite, there is a decomposition of $X$ into two disjoint sets $X_0$ and $X_1$ s.t. $X_A\subseteq X_0$ and $|X_0\setminus X_A|=|X_0|=|X_1|=|X|$. Let $F_i\vcentcolon=\langle X_i\rangle_F$ for all $i<2$. Then, by Lemma\cspace\ref{lemma characterization of f.f.}, $F=F_0\ast_V F_1$.
	
	\smallskip\noindent
	Now, let $\phi_0:F\rightarrow F_0$ be the isomorphism induced by a bijection $X\rightarrow X_0$ fixing $X_A$ pointwise (such a bijection exists by condition $|X_0\setminus X_A|=|X\setminus X_A|$ above). Then, since isomorphisms preserve $V$-free products, necessarily $\phi_0(A)=A\leq_*F_0$, and, for any complementary factor $C$ of $A$ in $F$, we have that:
	\begin{equation}\tag{$*_1$}\label{eq.1 - lemma existence of V-free complementary factor}
		F=F_0*_\varv F_1=A*_\varv C_0*_\varv F_1,
	\end{equation}
	with $C_0\vcentcolon=\phi_0(C)$.
	Equation\cspace\eqref{eq.1 - lemma existence of V-free complementary factor} witnesses that $C'\vcentcolon=C_0*_\varv F_1$ is a complementary factor of $A$ in $F$. To see that $C'$ is $\varv$-free, recall that $F_1$ is $\varv$-free with infinite rank greater or equal than $\mathrm{rank}(A)$. So, the basis $X_1$ admits a decomposition into two disjoint sets $X_A'$ and $X_1'$ s.t. $|X_A'|=\mathrm{rank}(A)$ and $|X_1'|=|X_1|$. As a result, we get two $\varv$-free subgroups $A_1\vcentcolon=\langle X_A'\rangle_F$ and $F_1'\vcentcolon=\langle X_1'\rangle_F$ of $F_1$ s.t. $A_1\cong A$, $F_1'\cong F$ and $F_1=A_1*_\varv F_1'$, by Lemma\cspace\ref{lemma characterization of f.f.}. 
	
	\smallskip\noindent
	Putting all this together, by the associativity of $V$-free product we obtain the following factorization:
	\begin{equation}\tag{$*_2$}\label{eq.2 - lemma existence of V-free complementary factor}
		C'=C_0*_\varv F_1=(C_0*_\varv A_1)*_\varv F'_1.
	\end{equation}
	Since the $\varv$-free product of $\varv$-free groups is still $\varv$-free and
	\begin{equation*}
		C_0*_\varv A_1\cong C_0\ast_V A=F_0\cong F,
	\end{equation*}
	\noindent
	equation \eqref{eq.2 - lemma existence of V-free complementary factor} shows that $C'$ is $\varv$-free, as desired.
\end{proof}
\noindent
We conclude this series of preliminary lemmas by recalling a simple, but crucial, consequence of Tarski-Vaught test for elementary substructures (see \cite{marker}, or \cite{ziegler-tent}, for a detailed discussion). The statement of Lemma\cspace\ref{lemma f.f. implies el. substructure} is already known as folklore. For the benefit of the reader, we give here a complete proof, which is just a straightforward generalization of Lemma\cspace 2.12 in \cite{perin} to the case of a generic variety of groups. 
\begin{fact}[Tarski-Vaught test]\label{fact - tarski-vaught}
	Suppose that $M$ is a substructure of a first-order structure $N$. Then, $M$ is an elementary substructure of $N$ if and only if for every formula $\phi(t,\bar{z})$ and $\bar{a}\in M^{|\bar{z}|}$, whenever there is some $b\in N$ s.t. $N\models\phi(b,\bar{a})$, then there is also a $c\in M$ for which $N\models\phi(c,\bar{a})$.
\end{fact}
\begin{lemma}\label{lemma f.f. implies el. substructure}
	Let $V$ be a variety of groups and $A,F$ be $V$-free groups of infinite rank. If $A\leq_*F$ and $\mathrm{rank}(A)<\mathrm{rank}(F)$, then $A$ is an elementary subgroup of $F$.
\end{lemma}
\begin{proof}
	Consider some formula $\phi(t,\bar{z})$ and a tuple $\bar{a}\in A^{|\bar{z}|}$ admitting an element $b\in F$ s.t. $F\models\phi(b,\bar{a})$. By Tarski-Vaught test (Fact.\cspace\ref{fact - tarski-vaught}), it suffices to show that there is some $c\in A$ s.t. $F\models\phi(c,\bar{a})$.
	
	\smallskip\noindent
	By the rank bound, Lemma\cspace\ref{lemma existence of V-free complementary factor} and Lemma\cspace\ref{lemma characterization of f.f.} ensure the existence of a $V$-basis $X_A$ of $A$ extending to a $V$-basis $X$ of $F$. Now, let $X_0$ be the finite subset of $X$ consisting of the elements occurring in the expression of $b$ as reduced word in $X$. Similarly, in the expression of the coordinates of $\bar{a}$ as reduced words in $X$ only a finite subset $X_1$ of $X_A$ is involved. So, since $|X_A|=\mathrm{rank}(A)$ is infinite, there exists a bijection $X\rightarrow X$ mapping $X_0$ into $X_A$ and fixing $X_1$ pointwise. Notice that the induced automorphism $\alpha\in\mathrm{Aut}(F)$ is s.t. $\alpha(b)\in A$ and $\alpha(\bar{a})=\bar{a}$. Moreover, it witnesses that $F\models\phi(\alpha(b),\alpha(\bar{a}))$, i.e. $F\models\phi(\alpha(b),\bar{a})$. Therefore, by letting $c\vcentcolon=\alpha(b)$, we get that $c\in A$ and $F\models\phi(c,\bar{a})$, as desired.
\end{proof}
\if{Exactly as in the case of free abelian groups and direct summands, the property of being $\varv$-free is preserved under quotienting by a $\varv$-free factor. This is the content of Lemma\cspace\ref{lemma quotient of V-free factor is V-free}.
\begin{notation}\label{notation S/N}
	In the rest of the paper we will adopt the following convention. If $A$ is a group, and $N$ is a normal subgroup of $A$, then, for any subset $S\subseteq A$, we let:
	\begin{equation*}
		S/N\vcentcolon=\{sN\,:\,s\in S\}.
	\end{equation*}  
\end{notation}
\noindent
Observe that $S/N$ is generally not a group. Indeed, in the cases of our interest, $N$ will typically be a verbal subgroup of $A$ (cf. Definition\cspace\ref{definition verbal subgroup}), and $S$ will be just a free basis of $A$.
\begin{lemma}\label{lemma quotient of V-free factor is V-free}
	Let $\varv$ be a variety of groups. Suppose that $A,B\in \varv$ are $\varv$-free groups such that:
	\begin{enumerate}[(1)]
		\item $B\leq_* A$;
		\item $B$ has a $\varv$-free complementary factor $C$ in $A$.
	\end{enumerate}
	Then, $A/N_A(B)$ is $\varv$-free. Furthermore, if $Z$ is a $\varv$-basis of $C$, then $Z/N_A(B)$ is a $\varv$-basis of $A/N_A(B)$.
\end{lemma}
\begin{proof}
	For ease of reading, we denote $N\vcentcolon= N_A(B)$ and by $\pi:A\rightarrow A/N$ the natural projection to the quotient. Let $Z$ be a $\varv$-basis of a complementary factor $C$ of $B$ in $A$. Since $A$ and $C$ are $\varv$-free and $A=B*_\varv C$, any $\varv$-basis $Y$ of $B$ naturally extends to a $\varv$-basis $X\vcentcolon=Y\sqcup Z$ of $A$. Clearly, $\pi(X)$ is a set of generators for $A/N$, as so is $X$ for $A$. Since $Y\subseteq N$, this means that $Z/N=\pi(Z)$ generates $A/N$. So, it only remains to show that the set $\pi(Z)$ satisfies the Universal Property of Remark\cspace\ref{remark: universal property}.
	 
	\smallskip\noindent
	 We can associate to any $\varv$-group $D$, and any map $\hat{f}:Z/N\rightarrow D$, a function ${f}:Y\sqcup Z\rightarrow D$ such that:
	\begin{equation*}
		{f}(x)=\begin{cases}
			\hat{f}(xN)\quad&\text{if }x\in Z;\\
			e\quad&\text{if }x\in Y.
		\end{cases}
	\end{equation*}
	Since $Y\sqcup Z$ is a $\varv$-basis of $A$, ${f}$ uniquely extends to a homomorphism ${\phi}:A\rightarrow D$. If $N\subseteq\ker({\phi})$, there exists a homomorphism $\hat{\phi}: A/N\rightarrow D$ s.t. ${\phi}=\hat{\phi}\circ\pi$ (cf. \cite{johnson}, Lemma\cspace1 pg.\cspace42). Moreover, if this is the case, the uniqueness of $\hat{\phi}$ follows directly from the uniqueness of ${\phi}$. We show that the inclusion $N\subseteq\ker({\phi})$ holds.
	
	\smallskip\noindent
	Every element of $N$ is a group word whose nested terms are of the form $aw(\bar{y})a^{-1}$, for some group word $w(\bar{z})$, $\bar{y}\in Y^{|\bar{z}|}$, and $a\in A$. On these terms, $\phi$ acts as follows:
	\begin{equation*}
		\begin{split}
			{\phi}(aw(\bar{y})a^{-1})&={\phi}(a)w({\phi}(\bar{y})){\phi}(a)^{-1}\\
			&={\phi}(a)w({f}(\bar{y})){\phi}(a)^{-1}\\
			&={\phi}(a){\phi}(a)^{-1}=e,
		\end{split}
	\end{equation*}
	since, by definition, ${\phi}(y)={f}(y)=e$ for all $y\in Y$, so that $w({f}(\bar{y}))$ is a group expression involving only trivial elements. Therefore, $N\subseteq\ker({\phi})$, and $\hat{\phi}$ is well-defined, as desired. 
\end{proof}}\fi
\noindent
Our strategy in Section\cspace4 will consist in constructing a \say{pathological} factorization pattern (with respect to the relation $\leq_*$) in the variety of all groups, and then transposing it into the desired one through a suitable functor. In the rest of this section we briefly outline the construction of this machinery.\if{The rest of this section is devoted to the construction of this machinery, showing that the functorial bridge preserves freeness and free decompositions among varieties.}\fi 
\begin{definition}\label{definition verbal subgroup}
	Let $A$ be a group and $\varv$ be a variety. The \emph{verbal subgroup} of $A$ corresponding to $\varv$ is the subgroup:
	\begin{equation*}
		\wordsv(A)\vcentcolon=\left\langle w(\bar{a})\,:\,w(\bar{z})\in\wordsv,\bar{a}\in A^{|\bar{z}|}\right\rangle_A.
	\end{equation*} 
\end{definition}
\noindent
Observe that verbal subgroups are fully characteristic, as they are invariant under all automorphisms. Hence, for every pair $A$ and $V$ as in Definition\cspace\ref{definition verbal subgroup}, the quotient $A/\wordsv(A)$ is well-defined.

\smallskip\noindent
Before going further, we recall some fundamental notions from Category Theory that we will need in the following.
\begin{notation}
	Let $\catC$ be a category. We denote by $\mathrm{Ob}(\catC)$ the class of the objects of $\catC$, and by $\mathrm{Hom}_\catC(X,Y)$ the set\footnote{All the categories considered in this paper are locally small.} of the morphisms of $\catC$ from $X$ to $Y$, for all $X,Y\in\mathrm{Ob}(\catC)$.
\end{notation}
\begin{definition}
	Let $\catC$ and $\catD$ be categories. A \emph{(covariant) functor} $F:\catC\rightarrow \catD$ is an assignment consisting of\begin{enumerate}[$\bullet$]
			\item a map $\mathrm{Ob}(\catC)\rightarrow\mathrm{Ob}(\catD)$ s.t. $X\mapsto F(X)$;
			\item a map $\mathrm{Hom}_\catC(X,Y)\rightarrow\mathrm{Hom}_\catD(F(X),F(Y))$ s.t. $f\mapsto F(f)$, for any pair $X,Y\in\mathrm{Ob}(\catC)$.
		\end{enumerate}
		These maps must preserve identities and compositions of morphisms, i.e. $F(\mathrm{id}_X)=\mathrm{id}_{F(X)}$, and $F(f\circ g)=F(f)\circ F(g)$, for all $X\in\mathrm{Ob}(\catC)$ and all composable morphisms $f,g$ of $\catC$.
		\if{ If $R:\catC\rightarrow \catD$ and $L:\catD\rightarrow \catC$ are functors, we say that $L$ is a \emph{left adjoint} of $R$ $($or, equivalently, that $R$ is a \emph{right adjoint} of $L)$ if, for every $A\in \mathrm{Ob}(\catC)$, and $B\in\mathrm{Ob}(\catD)$, there exists a bijection $\Phi_{B,A}:\mathrm{Hom}_\catC(L(B),A)\rightarrow\mathrm{Hom}_\catD(B,R(A))$  s.t.
		\begin{equation*}
			\Phi_{B',A'}(f \circ h \circ L(g)) = R(f) \circ \Phi_{B,A}(h) \circ g,
		\end{equation*}
		\smallskip\noindent
		for all $A'\in\mathrm{Ob}(\catC)$, $B'\in\mathrm{Ob}(\catD)$, $f\in\mathrm{Hom}_\catC(A,A')$, $h\in\mathrm{Hom}_\catC(L(B),A)$ and $g\in\mathrm{Hom}_\catD(B',B)$. In this case we also say that $\Phi_{B,A}$ is \emph{natural} in both $A$ and $B$.
	}\fi
\end{definition}
\noindent
If $\varv$ is a subvariety of a variety of groups $\varw$ (i.e. the class-inclusion $\varv\subseteq\varw$ holds), there is a well-defined (covariant) functor from $\cat(\varw)$ to $\cat(\varv)$ given by the assignments:
\begin{enumerate}[$\bullet$]
	\item $A\mapsto A_\varv \vcentcolon =A/\wordsv(A)$, for all objects $A\in \varw$;
	\item $f\mapsto f_\varv$, where $f_\varv$ is the homomorphism defined as
	\begin{equation*}
		f_\varv:A_\varv\rightarrow B_\varv \quad\text{s.t.}\quad a\wordsv(A)\mapsto f(a)\wordsv(B),
	\end{equation*}
	for all arrows (i.e. homomorphisms) $f:A\rightarrow B$ between objects $A$ and $B$ of $\varw$.
\end{enumerate}
\begin{notation}\label{definition functor W->V}
	If $\varv\subseteq\varw$ are varieties of groups, we call the functor $\cat(\varw)\rightarrow\cat(\varv)$ defined by the conditions above the \emph{verbal functor} induced by $\varv$ $($in $\varw$$)$, and we denote it by $\varv$ as well\if{\footnote{The \say{verbal functor} considered in\cspace\cite{pope1984} is the composition of this functor and the inclusion functor of $\cat(\varw)$ in $\cat(\varg)$. This generalization will reveal particularly useful in the following, as it will give a natural way to transfer group tower-constructions through arbitrary varieties.}}\fi.
\end{notation}
Technically speaking, the verbal functor of Notation\cspace\ref{definition functor W->V} depends on both $\varv$ and $\varw$. However, the expression of $A_\varv$ is the same both if we consider $A$ as an object of $\cat(\varw)$, or as an object of the category given by any other variety containing $\varv$ (e.g. $\varg$). A similar observation holds for the arrows of $\cat(\varw)$ as well. Therefore, we omit the dependency from $\varw$ in the notation and simply write $\varv$: in the following, the domain category will be always clear from the context.
\if{
\smallskip\noindent
The first interesting feature of verbal functors is that they preserve free groups among varieties. We formalize this statement in the following result.
\begin{lemma}\label{lemma verbal functor preserves free groups}
	Let $\varv\subseteq\varw$ be varieties of groups. If $A$ is a $\varw$-free group, then $A_\varv$ is $\varv$-free. Moreover, if $X$ is a $\varw$-basis of $A$, then the subset $X/\wordsv(A)$ of $A_\varv$ consisting of the cosets of the elements of $X$ is a $\varv$-basis of $A_\varv$.
\end{lemma}
\begin{proof}
	Let $X\subseteq A$ be a $\varw$-basis of $A$. It suffices to show that $X/\wordsv(A)$ and $A_\varv$ satisfy the Universal Property for $\varv$. For every group $B\in\varv\subseteq\varw$, and any function $\hat{f}:X/\wordsv(A)\rightarrow B$, the mapping ${f}:X\rightarrow B$ s.t. $x\mapsto \hat{f}(x\wordsv(A))$ induces a unique homomorphism ${\phi}:A\rightarrow B$ satisfying ${f}\subseteq{\phi}$, since $A$ is $\varw$-free by assumption.
	
	\smallskip\noindent
	Let $\hat{\phi}\vcentcolon=i\circ {\phi}_\varv$, where $i$ denotes the natural isomorphism between $B_\varv$ and $B$, i.e. $i:b\wordsv(B)\mapsto b$, for all $b\in B$ (recall that $B$ is a $\varv$-group, so the verbal subgroup $\wordsv(B)$ is trivial and such $i$ must be an isomorphism). We show that $\hat{\phi}:A_\varv\rightarrow B$ is the desired homomorphism.
	
	\smallskip\noindent
	Clearly, $\hat{\phi}$ extends $\hat{f}$, since we have that for all $x\in X$:
	\begin{equation*}
		\hat{\phi}(x\wordsv(A))=i({\phi}_\varv(x\wordsv(A)))=i(f(x)\wordsv(B))=f(x)=\hat{f}(x\wordsv(A)).
	\end{equation*}
	Finally, the uniqueness of $\hat{\phi}$ follows directly from the uniqueness of ${\phi}$, since any mapping of cosets in $X/\wordsv(A)$ uniquely yields a mapping of the basis $X$, as above.
\end{proof}
\begin{lemma}\label{lemma verbal functors preserve the rank of absolutely free groups}
	 Let $\varv$ be a non-trivial variety of groups (i.e. $\varv$ contains a group which is non-trivial). Then, the verbal functor $\varv:\cat(\varg)\rightarrow\cat(\varv)$ maps every absolutely free group into a $\varv$-free group of the same rank.
\end{lemma}
\begin{proof}
	Let $A$ be an absolutely free group with basis $X$. By Lemma \ref{lemma verbal functor preserves free groups}, $X/\wordsv(A)$ is a $\varv$-basis of $A_\varv$. We show that in this case the natural surjection $x\mapsto x\wordsv(A)$ is actually a bijection.
	
	\smallskip\noindent
	Towards a contradiction, assume that $x_1$ and $x_2$ are elements of $X$ whose cosets $x_1\wordsv(A)$ and $x_2\wordsv(A)$ are equal. Then, there is a group word $w(z_0,\ldots,z_{n-1})\in\wordsv$ s.t. $A\models x_1=x_2w(\bar{a})$, for some $\bar{a}=(a_0,\ldots,a_{n-1})$ in $A$. Without loss of generality, we can assume $w$ to be a a word in $X$. Indeed, since $X$ is a $\varv$-basis of $A$, there is a finite tuple $\bar{x}$ in $X$ s.t. each $a_i$ admits a unique expression in $A$ as group word: $a_i=v_i(\bar{x})$. Therefore, denoted the composed word $w(v_0(\bar{t}),\ldots,v_{n-1}(\bar{t}))$ simply by $w'(\bar{t})$, we have that:
		\begin{equation}\label{eq.1 - lemma verbal functors preserve the rank of absolutely free groups}\tag{$*_3$}
			A\models x_1=x_2w'(\bar{x}).
		\end{equation}
	Notice that since $w(\bar{z})\in\wordsv$, the expression $w'(\bar{b})= w(v_0(\bar{b}),\ldots,v_{n-1}(\bar{b}))$ is trivial in $A$, for all $\bar{b}\in A^{|\bar{t}|}$. So, $w'(\bar{t})\in\wordsv$ as well.
	
	\smallskip\noindent
	Finally, from expression \eqref{eq.1 - lemma verbal functors preserve the rank of absolutely free groups}, we get that the identity $w'(\bar{x})=x_2^{-1}x_1$ holds in $A$. Since $A$ is absolutely free with basis $X$, this means that $w'(\bar{t})$ is actually equivalent to the word $t_2^{-1}t_1$. However, in every group $B$ admitting $t_2^{-1}t_1$ as law, all the elements must be equal, so that $B=\{e\}$. Since $\varv$ is supposed to be non-trivial, we have an absurd, as desired.
\end{proof}
\begin{lemma}\label{lemma if A absolutely free then Awv is isomorphic to Av}
	Let $\varv\subseteq\varw$ be varieties of groups. If $A$ is an absolutely free group, then $(A_\varw)_\varv\cong A_\varv$.
\end{lemma}
\begin{proof}
	Observe that $\wordsv(A)$ and $\wordsw(A)$ are two normal subgroups of $A$ s.t. $\wordsw(A)\leq\wordsv(A)$. So, the quotient map $\pi:A\rightarrow A_\varv$ factors as $\pi=\pi_2\circ\pi_1$, with $\pi_1:A\rightarrow A_\varw$ and $\pi_2:A_\varw\rightarrow A_\varw/\pi_1(\wordsv(A))$.
		
		\smallskip\noindent
		We claim that $\pi_1(\wordsv(A))=\wordsv(A_\varw)$. To see inclusion $\wordsv(A_\varw)\subseteq\pi_1(\wordsv(A))$, consider a word $w(z_0,\ldots,z_{n-1})\in\wordsv$ and some coset $a_i\wordsw(A)$, for all $i<n$. Since $a_i\wordsw(A)=\pi_1(a_i)$ for any $i<n$, by definition of the group structure of $A_\varw$, we have that $w(a_0\wordsw(A),\ldots,a_{n-1}\wordsw(A))=\pi_1(w(a_0,\ldots,a_{n-1}))$, witnessing that $w(a_0\wordsw(A),\ldots,a_{n-1}\wordsw(A))\in\pi_1(\wordsv(A))$ as well. For inclusion $\pi_1(\wordsv(A))\subseteq\wordsv(A_\varw)$, take a word $w(z_0,\ldots,z_{n-1})\in\wordsv$ and some $a_i\in A$, for all $i<n$. Then,
		\begin{equation*}
			\pi_1(w(a_0,\ldots,a_{n-1}))=w(a_0,\ldots,a_{n-1})\wordsw(A),
		\end{equation*} which, by the group structureof $A_\varw$, is equal to $w(a_0\wordsw(A),\ldots,a_{n-1}\wordsw(A))$. Hence, we have that $\pi_1(w(a_0,\ldots,a_{n-1}))\in\wordsv(A_\varw)$.
	
	\smallskip\noindent
	Finally, from equality $\pi_1(\wordsv(A))=\wordsv(A_\varw)$, we get that:
\begin{equation*}
	\begin{split}
		\pi(A)&=\pi_2\circ\pi_1(A)=\pi_2(A_\varw)\\
		&=A_\varw/\pi_1(\wordsv(A))=A_\varw/\wordsv(A_\varw)\\
		&=(A_\varw)_\varv,
	\end{split}
\end{equation*}
as desired.
\end{proof}
\begin{lemma}\label{lemma if V C W, then V -| of the inclusion functor}
	In the context of {\normalfont Notation \ref{definition functor W->V}}, the verbal functor $\varv:\cat(\varw)\rightarrow\cat(\varv)$ is the left adjoint of the inclusion functor of $\cat(\varv)$ in $\cat(\varw)$.
\end{lemma}
\begin{proof}
	Let $\inc:\cat(\varv)\rightarrow\cat(\varw)$ be the inclusion functor of $\cat(\varv)$ in $\cat(\varw)$ (i.e. $\inc$ acts like the identity on every object and arrow in $\cat(\varv)$). By definition, $\varv$ is the left adjoint of $\inc$ if and only if, for all $B\in\varw$ and $A\in\varv$, there exists a bijection between $\Hom_{\cat(\varv)}(B_\varv, A)$ and $\Hom_{\cat(\varw)}(B, I(A))$ which is natural in both $A$ and $B$.
	
	\smallskip\noindent
	Let $\Phi_{B,A}:\Hom_{\cat(\varv)}(B_\varv, A)\rightarrow\Hom_{\cat(\varw)}(B, A)$ be the function that associates to each homomorphism $\theta:B_\varv \rightarrow A$ a map 
	\begin{equation}\tag{$*_4$}\label{definition Phi_B,A(f)}
		\Phi_{B,A}(\theta): B \rightarrow A, \quad\text{s.t.}\quad b\mapsto \theta(b\wordsv(B)).
	\end{equation}
	We show that $\Phi_{B,A}$ is as required.
	
	\smallskip\noindent
	Clearly, any $\Phi_{B,A}(\theta)$ as in \eqref{definition Phi_B,A(f)} is a well-defined function which preserves the identity. The compatibility with the group law is provided by the one of $\theta$. Indeed, for all $a,b\in B$,
	\begin{equation*}
		\begin{split}
			\Phi_{B,A}(\theta)(ab)&=\theta(ab\wordsv(B))\\
			&=\theta(a\wordsv(B)b\wordsv(B))\\
			&=\theta(a\wordsv(B))\theta(b\wordsv(B))\\
			&=\Phi_{B,A}(\theta)(a)\Phi_{B,A}(\theta)(b),
		\end{split}
	\end{equation*}
	where the second equality follows by the group structure on $B_\varv$, and the third one by the fact that $\theta$ is a homomorphism. Therefore, $\Phi_{B,A}$ is well-defined.
	\begin{2.20.1claim}\label{claim Phi bijection}
		$\Phi_{B,A}$ is a bijection.
	\end{2.20.1claim}
	\begin{proof}[{\normalfont \proofclaim}\cspace2.20.1]
		The injectivity is immediate: if $\theta,\eta:B_\varv\rightarrow A$ are homomorphisms for which $\Phi_{B,A}(\theta)=\Phi_{B,A}(\eta)$, then,
		\begin{equation*}
			\theta(b\wordsv(B))=\Phi_{B,A}(\theta)(b)=\Phi_{B,A}(\eta)(b)=\eta(b\wordsv(B)),
		\end{equation*}
		for all $b\in B$, i.e. $\theta=\eta$.
		
		\smallskip\noindent
		For the surjectivity, consider an arbitrary homomorphism $\sigma:B\rightarrow  A$. We show that $\sigma$ belongs to the image of $\Phi_{B,A}$. Let $\tilde{\sigma}_\varv\vcentcolon=i\circ \sigma_\varv$, where $i$ is the isomorphism between $A_\varv$ and $A$ given by the assignment $a\wordsv(A)\mapsto a$, for all $a\in A$ (recall that $A$ is a $\varv$-group, so $\wordsv(A)$ is trivial). Then, by definition,
		\begin{equation*}
			\begin{split}
				\Phi_{B,A}(\tilde{\sigma}_\varv)(b)&=\tilde{\sigma}_\varv(b\wordsv(B))\\
				&=i(\sigma_\varv(b\wordsv(B)))\\
				&=i(\sigma(b)\wordsv(A))=\sigma(b),
			\end{split}
		\end{equation*}
		for all $b\in B$. Therefore, $\sigma=\Phi_{B,A}(\tilde{\sigma}_\varv)$ and $\Phi_{B,A}$ is a bijection, as desired.
	\end{proof}
	\begin{2.20.2claim}\label{claim Phi natural}
		$\Phi_{B,A}$ is natural in both $A$ and $B$.
	\end{2.20.2claim}
	\begin{proof}[{\normalfont \proofclaim}\cspace2.20.2]
		We just need to show that, if $A,A'\in \varv$ and $B,B'\in \varw$, then, for all homomorphisms $\alpha:A\rightarrow A'$ and $\beta:B'\rightarrow B$, the following diagram commutes.
		\begin{equation}\tag{$*_5$}\label{eq natural transformation}
			\begin{tikzcd}
				{\Hom_{\cat(\varv)}(B_\varv,A)} \arrow[r, "{\Phi_{B,A}}", rightarrow] \arrow[d, "{\Hom_{\cat(\varv)}(\beta_\varv,\alpha)=\alpha\circ(\,\cdot \,)\circ \beta_\varv}"', rightarrow] & {\Hom_{\cat(\varw)}(B,A)} \arrow[d, "{\Hom_{\cat(\varw)}(\beta,\alpha)=\alpha\circ(\,\cdot\,)\circ \beta}", rightarrow] \\
				{\Hom_{\cat(\varv)}(B'_\varv,A')} \arrow[r, "\Phi_{B',A'}"', rightarrow] & {\Hom_{\cat(\varw)}(B',A')}
			\end{tikzcd}
		\end{equation}
		To this extent, let $\sigma:B_\varv\rightarrow A$ be a homomorphism, and $b\in B'$. Then,
		\begin{equation*}
			\begin{split}
				\Phi_{B',A'}(\alpha\circ \sigma\circ \beta_\varv)(b)&=(\alpha\circ \sigma\circ \beta_\varv)(b\wordsv(B'))\\
				&=\alpha(\sigma(\beta(b)\wordsv(B)))\\
				&=\alpha(\Phi_{B,A}(\sigma)(\beta(b)))=(\alpha\circ \Phi_{B,A}(\sigma)\circ \beta)(b).
			\end{split}
		\end{equation*}
		Therefore, \eqref{eq natural transformation} commutes, as desired.
	\end{proof}
\end{proof}
\noindent
Since left-adjoints preserve coproducts\footnote{The Coproduct of $\kappa$ elements is just a colimit over the diagram consisting of $\kappa$ objects, and left-adjoints preserve colimits (see\cspace\cite{riehl}, Theorem\cspace4.5.3, pg.\cspace138 for further details).}, a direct consequence of Lemma \ref{lemma if V C W, then V -| of the inclusion functor} is that the relation $\leq_*$ is \say{invariant} under the action of the verbal functor $\varv$. The following corollary formalizes this observation.
\begin{corollary}\label{fact verbal functors preserve free factors}
	Let $\varv\subseteq \varw$ be varieties of groups. If $A$ and $B$ are $\varw$-groups s.t. $A\leq_* B$, then $A_\varv\leq_*B_\varv$\footnote{Technically speaking, what we are claiming is that there exists an isomorphic copy $A_\varv'$ of $A_\varv$ in $B_\varv$ s.t. $A_\varv'\leq_* B_\varv$, in the sense of Definition \ref{definition V-free product}.}.
\end{corollary}
\noindent
Corollary\cspace\ref{fact verbal functors preserve free factors} provides the last technical tool we need in the proof of our main result, Theorem\cspace\ref{theorem V satisfies (CP) -> Fv(k) not homoheneous}. Nevertheless, to comprehensively conclude this section, we recall the following lemma, originally stated by Pope (see Lemma\cspace7, pg.\cspace43 of \cite{pope1984}), for which we give here a complete proof. It will reveal particularly useful in pursuing Strategy\cspace\ref{strategy} (i.e. in the application of Theorem\cspace\ref{theorem V satisfies (CP) -> Fv(k) not homoheneous} to generic varieties of groups), as it provides an easy criterion to check that a subgroup $B\leq A$ is not a $\varv$-free factor of $A\ast_\varv F_\varv(\aleph_0)$ by looking at the quotient $C=A/N_A(B)$.
\begin{lemma}\label{lemma pope}
	Let $\varv$ be a variety of groups. Suppose that $A$ and $B$ are $\varv$-free groups s.t. $B\leq A$. Then, for every group $D\in\varv$ (not necessarily $\varv$-free), denoted $C=A/N_A(B)$ and $A'=A\ast_\varv D$, there is an isomorphism $C\ast_\varv D\cong A'/N_{A'}(B)$.
\end{lemma}
\begin{proof}
	For simplicity of notation, we let $N\vcentcolon=N_A(B)$ and $N'\vcentcolon=N_{A'}(B)$. The plan is to define the desired isomorphism $\phi:C\ast_\varv D\rightarrow A'/N'$ through the Diagram Property of $\varv$-free products. To this extent, we consider the homomorphism $\gamma:C\rightarrow A'/N'$ s.t. $\gamma(aN)=aN'$ for all $aN\in C$. Since $C=A/N$ and $N\leq N'$, such a $\gamma$ is clearly well-defined. Similarly, we define a homomorphism $\delta:D\rightarrow A'/N'$ s.t. $\delta(d)=dN'$ for all $d\in D$. Notice that $\delta$ is just the composition of the natural inclusion of $D$ in $A'$ with the canonical projection $\pi':A'\rightarrow A'/N'$. So, also $\delta$ is a well-defined homomorphism. Therefore, by the Diagram Property of Definition\cspace\ref{definition V-free product}, there is a unique group homomorphism $\phi:C\ast_\varv D\rightarrow A'/N'$ s.t. $\phi\restriction_C\,=\gamma$ and $\phi\restriction_D\,=\delta$.
	
	\smallskip\noindent
	Now, let $\widetilde{\pi}:A\rightarrow C\ast_\varv D$ be the composition of the canonical projection $\pi:A\rightarrow C$ and the inclusion of $C$ in $C\ast_\varv D$, and $i_D:D\rightarrow C\ast_\varv D$ be the natural inclusion of $D$ in $C\ast_\varv D$. Then, we denote by $\psi: A'\rightarrow C\ast_\varv D$ the unique group homomorphism s.t. $\psi\restriction_A\,=\widetilde{\pi}$ and $\psi\restriction_D\,=i_D$, whose existence is ensured by the Diagram Property of $\varv$-free products. By construction, the following equality holds:
	\begin{equation}\tag{$*_6$}\label{eq.1 - lemma pope}
		\phi\circ\psi=\pi',
	\end{equation}
	with $\pi':A'\rightarrow A'/N'$ denoting as above the canonical projection of $A'$ onto $A'/N'$. Indeed, equality \eqref{eq.1 - lemma pope} is realized by the restrictions of $\phi\circ \psi$ and $\pi'$ to both $A$ and $D$, so by the Diagram Property of $\varv$-free products it holds on the whole $A'=A\ast_\varv D$.
	
	\smallskip\noindent
	Finally, since $N'\leq \mathrm{ker}(\psi)$ (cf.\cspace\cite{johnson}, Lemma\cspace1 pg.\cspace42), there is a group homomorphism $\sigma: A'/N'\rightarrow C\ast_\varv D$ such that:
	\begin{equation}\tag{$*_7$}\label{eq.2 - lemma pope}
		\sigma\circ\pi'=\psi.
	\end{equation}
	We claim that $\phi$ and $\sigma$ are mutually inverse. Indeed, conditions \eqref{eq.1 - lemma pope} and \eqref{eq.2 - lemma pope} ensure that:
	\begin{equation*}
		(\phi\circ\sigma)(a'N')=(\phi\circ\sigma)(\pi'(a'))=\phi(\psi(a'))=\pi'(a')=a'N',
	\end{equation*}
	for all $a'\in A'$, witnessing that $\phi\circ\sigma=\mathrm{id}_{A'/N'}$. Conversely, in order to prove that $\sigma\circ\phi=\mathrm{id}_{C\ast_\varv D}$, it is enough to show that $(\sigma\circ\phi)\restriction_C\,=\mathrm{id}_{C}$ and $(\sigma\circ\phi)\restriction_D\,=\mathrm{id}_{D}$. If this is the case, the Diagram Property of $\varv$-free products ensures the desired equality. So, for all $a\in A$, by construction of $\phi$, we have:
	\begin{equation*}
		(\sigma\circ\phi)(aN)=\sigma(aN')=\sigma(\pi'(a))=\psi(a)=aN,
	\end{equation*}
	where the last equality follows by condition $\psi\restriction_A\,=\widetilde{\pi}$. Similarly, from $\psi\restriction_D\,=i_D$ it follows that:
	\begin{equation*}
		(\sigma\circ\phi)(d)=\sigma(dN')=\sigma(\pi'(d))=\psi(d)=d,
	\end{equation*}
	 for all $d\in D$, as desired. Therefore, $\sigma=\phi^{-1}$ and $\phi:C\ast_\varv D\rightarrow A'/N'$ is the desired isomorphism.
\end{proof}}\fi
\section{The proof of the Main Theorem}
The construction principle, or simply $(\cp)$, is a combinatorial condition on countable free algebras characterizing those varieties $\varv$ in a countable language containing an $L_{\infty\omega_1}$-free algebra (i.e. $L_{\infty\omega_1}$-equivalent to some $\varv$-free algebra) of cardinality $\aleph_1$ which is not itself $\varv$-free. Such condition, first introduced in \cite{eklof}, has been extensively studied for varieties of groups in \cite{mekler_groups}. In the current section, we will show its relevance for the analysis of the homogeneity of $\varv$-groups of uncountable rank.
\begin{definition}\label{definition CP}
	Let $\varv$ be a variety of groups. We say that the \emph{construction principle} $(\cp)$ holds in $\varv$ if there are two $\varv$-free groups $B\leq A$ or rank $\aleph_0$ and a $\varv$-basis $Y=\{y_i:i<\omega\}$ of $B$ s.t.
	\begin{enumerate}[(1)]
		\item $B_n=\langle \{y_i:i\leq n\}\rangle_A\leq_* A*_\varv F_\varv(\aleph_0)$, for all $n<\omega$;
		\item $B$ is not a $\varv$-free factor of $A*_\varv F_\varv(\aleph_0)$.
	\end{enumerate}
\end{definition}
\noindent
Notice that, as consequence of Lemma\cspace\ref{lemma existence of V-free complementary factor} and Lemma\cspace\ref{lemma characterization of f.f.}, $(\cp)$ also provides some additional information:
\begin{enumerate}[(i)]
	\item for all $n<\omega$, $B_n$ has a $\varv$-free complementary factor in $A*_\varv F_\varv(\aleph_0)$;
	\item $B$ has no $V$-free complementary factor in $A*_\varv F_\varv(\aleph_0)$.
\end{enumerate}
In particular, observation $(\mathrm{ii})$ suggests the proof of the following result.
\begin{1.1theorem}
	Let $V$ be a variety of groups. If $(\cp)$ holds in $\varv$, then for every uncountable $\kappa$, $F_V(\kappa)$ is not $\aleph_1$-homogeneous. Furthermore, $F_V(\aleph_0)$ has an elementary subgroup which is not a $V$-free factor of $F_V(\aleph_0)$.
\end{1.1theorem}
\begin{proof}[Proof of \ref{theorem V satisfies (CP) -> Fv(k) not homoheneous}]
	For simplicity of notation, we denote $F=F_\varv(\kappa)$. Let $A$ and $B$ be $\varv$-free groups witnessing that $(\cp)$ holds in $\varv$, so that there is a $\varv$-basis $Y=\{y_i:I<\omega\}$ of $B$ s.t. $B_n=\langle \{y_i:i\leq n\}\rangle\leq_*A*_\varv F_\varv(\aleph_0)$, for all $n<\omega$. Without loss of generality, we can assume $A*_\varv F_\varv(\aleph_0)$ to embed in $F$ in the natural way. Hence, every $B_n$ is also a $\varv$-free factor of $F$, by the transitivity of relation $\leq_*$. Moreover, since $F$ has uncountable rank, Lemma\cspace\ref{lemma existence of V-free complementary factor} ensures that $A$ and every $B_n$ have some complementary factors in $F$ which are $\varv$-free. Therefore, by Lemma\cspace\ref{lemma characterization of f.f.} we can assume $X'=\{x_i:i<\kappa\}$ to be a $\varv$-basis of $F$ s.t. the initial segment $X=\{x_i:i<\omega\}$ is a $\varv$-basis of $A$.
	
	\smallskip\noindent
	Let $\bar{x}=(x_i:i<\omega)$ and $\bar{y}=(y_i:i<\omega)$ be enumerations of $X$ and $Y$, respectively. We claim that $\bar{x}$ and $\bar{y}$ have the same type in $F$, but there is no automorphism of $F$ identifying them. Indeed, we have the following.
	\begin{1.1.1claim}
		$\tp^{F}(\bar{x})=\tp^{F}(\bar{y})$.
	\end{1.1.1claim}
	\begin{proof}[\proofclaim]
		Consider a first order formula $\phi(z_0,\ldots,z_{n})$ of the language of groups $L$. Since $B_n=\langle \{y_i:i\leq n\}\rangle_A$ is a $\varv$-free factor 
		of $F$ with $\varv$-free complementary factor, by Lemma\cspace\ref{lemma characterization of f.f.}, the $\varv$-basis $\{y_i:i\leq n\}$ of $B_n$ extends to a $\varv$-basis $Z$ of $F$. Therefore, by Remark\cspace\ref{remark bijection of V-bases yields an automorphism of a V-free algebra}, any bijection $s:X'\rightarrow Z$ s.t. $s(x_i)=y_i$, for all $i\leq n$, naturally induces an automorphism $\sigma\in \mathrm{Aut}(F)$ s.t. $s\subseteq \sigma$. So, $\sigma$ witnesses that:
		\begin{equation*}
			F\models \phi(x_0,\ldots,x_{n})\quad\text{if and only if}\quad F\models\phi(y_0,\ldots,y_{n}),
		\end{equation*}	
		as desired.
	\end{proof}
	\noindent
	Towards a contradiction, suppose now that there is an automorphism $\alpha\in\mathrm{Aut}(F)$ s.t. $\alpha(\bar{x})=\bar{y}$. Then, since $A\leq_*F$, also $B=\alpha(A)\leq_*F$, as automorphisms preserve $\varv$-free factorizations. By Lemma\cspace\ref{lemma existence of V-free complementary factor}, $B$ admits a complementary factor in $F$ which is $\varv$-free of rank $\kappa$. Hence, by Lemma\cspace\ref{lemma characterization of f.f.} there is a $\varv$-basis $Y'$ of $F$ s.t. $Y\subseteq Y'$. So, we have the following decompositions:
	\begin{equation}\tag{$*_3$}\label{eq.1 - theorem V satisfies (CP) -> Fv(k) not homoheneous}
		F=B*_\varv\langle Y'\setminus Y\rangle_F=A*_\varv\langle X'\setminus X\rangle_F.
	\end{equation}
	\noindent
	Notice that both $B$ and the language $L$ are countable. So, in the expression of the elements of $B$ as reduced $L\,$-words only countably many elements of the $\varv$-basis $X'$ occur. Similarly, for the expression of those of $A$ as reduced $L\,$-words in $Y'$.  
	We will now use this observation to \say{refine} from the decompositions in \eqref{eq.1 - theorem V satisfies (CP) -> Fv(k) not homoheneous}
	a common countable $C\leq_*F$ s.t. both $A$ and $B$ are $\varv$-free factors of $C$.  
	
	\smallskip\noindent
	We proceed by induction on $n<\omega$ as follows. For $n=0$, we let $C_0\vcentcolon=A$. Suppose now that $C_n$ has been constructed, then we distinguish two cases:
	\begin{enumerate}[(i)]
		\item if $n$ is even, we let $C_{n+1}\vcentcolon=\langle Y_n \rangle_F$, and let $Y_n$ be the set of all $y\in Y'$ involved in the expression of some element of $C_n$ as reduced $Y'$-word, i.e. the set of all $y_0\in Y'$ for which there is a $c\in C_n$ and a reduced $L\,$-word $w_c(t_0,\bar{t})$ s.t. $F\models c=w_c(y_0,\bar{y})$, for some $\bar{y}\in Y'^{|\bar{t}|}$;
		\item if $n$ is odd, we let $C_{n+1}\vcentcolon=\langle X_n \rangle_F$, and, as before, we let $X_n$ be the set of all $x\in X'$ involved in the expression of some element of $C_n$ as reduced $X'$-word, i.e. the set of all $x_0\in X'$ for which there is a $c\in C_n$ and a reduced $L\,$-word $v_c(t_0,\bar{t})$ s.t. $F\models c=v_c(x_0,\bar{x})$, for some $\bar{x}\in X'^{|\bar{t}|}$.		
	\end{enumerate}
	Finally, we consider the union $C\vcentcolon=\bigcup_{n<\omega}C_n$. Since $C$ is union of the $\leq_*$-chain of $\varv$-free groups $(C_{2n+1}:n<\omega)$, it is $\varv$-free with $\varv$-basis $\widetilde{Y}\vcentcolon=\bigcup_{n<\omega}Y_{2n}$. Indeed, any element $c\in C$ already figures at a finite stage $2n+1$ of the construction, so there is a finite tuple $\bar{y}$ in $Y_{2n}$ s.t. $F\models c=w_c(\bar{y})$, for some reduced $L\,$-word $w_c(\bar{t})$. Further, this expression is unique, as $Y_{2n}$ is a subset of the $\varv$-basis $Y'$ of $F$. Similarly, $C$ is the union of the $\leq_*$-chain $(C_{2n}:n<\omega)$. So, also $\widetilde{X}\vcentcolon=\bigcup_{n<\omega}X_{2n+1}$ is a $\varv$-basis of $C$.
	
	\smallskip\noindent
	By construction, the following inclusions hold:
	\begin{equation*}
		Y\subseteq \widetilde{Y}\subseteq Y' \quad\text{and}\quad X\subseteq \widetilde{X}\subseteq X'.
	\end{equation*}
	Therefore, by Lemma\cspace\ref{lemma characterization of f.f.}, both $A$ and $B$ are $\varv$-free factors of $C$ with a $\varv$-free complementary factor. Let then $D$ be a $\varv$-free complementary factor of $A$ in $C$, so that we have:
	\begin{equation*}
		B\leq_*C\leq_*F \quad\text{and}\quad C=A*_\varv D.
	\end{equation*}
	Without loss of generality, we can assume $D$ to have rank $\aleph_0$. Indeed, $F$ is supposed to have uncountable rank, while, by construction, $C$ is a $\varv$-free factor of $F$ of countable rank. So, if $\mathrm{rank}(D)<\aleph_0$, it suffices to take a subset $Z$ of $X'\setminus\widetilde{X}$ s.t. $|Z|=\aleph_0$ and consider the $\varv$-free product $D'\vcentcolon=D*_\varv\langle Z\rangle_F$ in $F$. The resulting group still satisfies $B\leq_*A*_\varv D'\leq_*F$.
	
	\smallskip\noindent 
	Finally, by the Diagram Property of Definition\cspace\ref{definition V-free product}, $\mathrm{id}_A$ and the natural correspondence $D\rightarrow F_\varv(\aleph_0)$ uniquely determine an isomorphism $\beta:C\rightarrow A\ast_\varv F_\varv(\aleph_0)$ which fixes $B$ and witnesses that $B\leq_* \beta(A*_\varv D)=A*_\varv F_\varv(\aleph_0)$. However, this is absurd by condition \emph{(2)} of $(\cp)$.
	
	\smallskip\noindent
	It only remains to show that $F_\varv(\aleph_0)$ has an elementary subgroup which is not a $\varv$-free factor of $F_\varv(\aleph_0)$. Since $A$ has infinite countable rank, we can identify $F_\varv(\aleph_0)$ with $A$. Hence, if $B$ is an elementary substructure of $A$, then we have the thesis, since the previous argument shows that $B$ is not a $\varv$-free factor of $A$. We claim that this is the case.
	
	\smallskip\noindent
	Observe that $\tp^{A}(\bar{x})=\tp^{B}(\bar{y})$, as both $A$ and $B$ are $\varv$-free groups of the same rank, with $\varv$-free bases $\bar{x}$ and $\bar{y}$, respectively. Further, Lemma\cspace\ref{lemma f.f. implies el. substructure} shows that $A$ is elementary in $F$, since $A$ and $F$ have infinite ranks s.t. $\mathrm{rank}(A)<\mathrm{rank}(F)$, and $A$ embeds in $F$ as a $\varv$-free factor. In particular, $\tp^{A}(\bar{x})=\tp^{F}(\bar{x})$ and $\tp^A(\bar{y})=\tp^{F}(\bar{y})$. So, \emph{Claim\cspace1.1.1} completes the chain of equalities:
	\begin{equation}\tag{$*_4$}\label{eq.2 - theorem V satisfies (CP) -> Fv(k) not homoheneous}
		\tp^A(\bar{y})=\tp^{F}(\bar{y})=\tp^{F}(\bar{x})=\tp^{A}(\bar{x})=\tp^{B}(\bar{y}).
	\end{equation}
	In light of \eqref{eq.2 - theorem V satisfies (CP) -> Fv(k) not homoheneous}, the proof is straightforward. Let $\phi(z_0,\ldots,z_{n-1})$ be a formula of the language of groups, and $b_0,\ldots,b_{n-1}\in B$. Since $Y$ is a $\varv$-basis of $B$, each $b_i$ is labelled by a suitable group word $w_i(y_0,\ldots,y_{k-1})$ in $\bar{y}$. Hence, by replacing any occurrence of the $b_i$'s in $\phi$ with the corresponding group words $w_i$'s, we get:
	\begin{equation*}
		\begin{split}
			A\models\phi(b_0,\ldots,b_{n-1}) &\Leftrightarrow A\models\phi(w_0(y_0,\ldots,y_{k-1}),\ldots,w_{n-1}(y_0,\ldots,y_{k-1}))\\
			&\Leftrightarrow B\models\phi(w_0(y_0,\ldots,y_{k-1}),\ldots,w_{n-1}(y_0,\ldots,y_{k-1}))\\
			&\Leftrightarrow B\models\phi(b_0,\ldots,b_{n-1}),
		\end{split}		
	\end{equation*}
	where the second equivalence follows from \eqref{eq.2 - theorem V satisfies (CP) -> Fv(k) not homoheneous} and the fact that $\phi(w_0(y_0,\ldots,y_{k-1}),$\\$\ldots,w_{n-1}(y_0,\ldots,y_{k-1}))$ is a formula involving only terms of $\bar{y}$. Therefore, $B$ is an elementary substructure of $A$, as desired.
\end{proof}
\section{Applications}
In the previous section, we proved that in any variety $\varv$ satisfying $(\cp)$ every $\varv$-free group of uncountable rank is not $\aleph_1$-homogeneous. We now outline a method (i.e. Strategy\cspace\ref{strategy}) to show in practice that a given variety of groups satisfies $(\cp)$. To this extent, the main tool is provided by the following result, due to Mekler (see Lemma\cspace7, pg.\cspace135 of\cspace\cite{mekler_groups}).
\begin{theorem}[Mekler]\label{application theorem}
	Suppose that there are $A, B, Y$ and $\varv$ such that:
	\begin{enumerate}[(1)]
		\if{$\varv$ is a non-trivial variety of groups (i.e. $\varv$ contains a non-trivial group);}\fi
		\item $B\leq A$ are absolutely free groups of rank $\aleph_0$; 
		\item $Y = \{y_i : i < \omega \}$ is a basis of $B$;
		\item for every $n < \omega$, $B_n = \langle \{y_i : i \leq n \}\rangle_B \leq_* A$ $($where $\leq_*$ denotes \say{free factor}\,$)$;
		\item letting $C = A/N_{A}(B)$, we have that $C_V*_\varv F_\varv(\aleph_0)$ is not $V$-free.
	\end{enumerate}
	Then, any variety $\varw \supseteq \varv$ satisfies $(\cp)$.
\end{theorem}
\if{Intuitively, the proof of Theorem\cspace\ref{application theorem} can be read as follows. Assumption\cspace\emph{(3)} provides the countable {\color{purple}tuple we shall compare with a basis $X$ of $A$ (more properly, we will study the images of $X$ and $Y$ through the verbal functor $\varw:\mathrm{Cat}(\mathrm{Grp})\rightarrow\mathrm{Cat}(\varw)$).} Assumption\cspace\emph{(4)} ensures that they have the same type in $F_\varw(\kappa)$, since the truth of a given first order formula can be checked at a finite stage of the decomposition of $B_\varw$ as an increasing $\leq_*$-chain $\{(B_n)_\varw:n<\omega\}$. Finally, condition\cspace\emph{(5)} gives the absurd, showing that this process of comparison fails in the limit.}\fi
\if{\begin{proof}
	Let $A, B, Y, V$ as in \emph{(1)-(3)}, and $X$ be a basis of $A$. In light of Lemma\cspace\ref{lemma verbal functor preserves free groups} and Lemma\cspace\ref{lemma verbal functors preserve the rank of absolutely free groups}, it is not restrictive to identify $X$ and $Y$ with their images in $A_\varw$ and $B_\varw$, respectively. So, up to isomorphisms we have that $B_\varw\leq A_\varw$ and $B_\varw=\bigcup_{n<\omega}(B_n)_\varw$ for any variety $\varw\supseteq \varv$ as above. Moreover, by Corollary\cspace\ref{fact verbal functors preserve free factors} assumption\cspace\emph{(4)} ensures that condition\cspace\emph{(1)} of $(\cp)$ (cf. Definition\cspace\ref{definition CP}) is satisfied by $A_\varw, B_\varw$ and $Y$.
	
	\smallskip\noindent
	Towards a contradiction, suppose now that:
	\begin{equation}\tag{$*_{10}$}\label{eq.1 - application theorem}
		B_\varw\leq_*A_\varw\ast_\varw F_\varw(\aleph_0).
	\end{equation}
	Since $\varv\subseteq\varw$, the verbal functor $\varv:\text{Cat}(\varw)\rightarrow\text{Cat}(\varv)$ from Notation\cspace\ref{definition functor W->V} is well-defined. Up to isomorphisms, it maps decomposition \eqref{eq.1 - application theorem} in $B_\varv\leq_*A_\varv\ast_\varv F_\varv(\aleph_0)$. By Lemma\cspace\ref{lemma existence of V-free complementary factor}, $B_\varv$  has a $\varv$-free complementary factor in $A_\varv\ast_\varv F_\varv(\aleph_0)$.	Hence, denoted $A'\vcentcolon=A_\varv\ast_\varv F_\varv(\aleph_0)$, the quotient $A'/N_{A'}(B_\varv)$ is $\varv$-free, by Lemma\cspace\ref{lemma quotient of V-free factor is V-free}. However, by Lemma\cspace\ref{lemma pope} there is an isomorphism:
	\begin{equation*}
		A'/N_{A'}(B)\cong C_\varv\ast_\varv F_\varv(\aleph_0),
	\end{equation*}
	which is absurd, as $C_\varv\ast_\varv F_\varv(\aleph_0)$ is supposed to be not $\varv$-free. Therefore, $\varw$ satisfies also condition\cspace\emph{(2)} of $(\cp)$, as desired.
\end{proof}}\fi
\noindent
Theorem\cspace\ref{application theorem} is a very compact machinery: it takes as an input a tuple $(A,B,Y,\varv)$ satisfying assumptions \emph{(1)-(5)}, and \say{produces} as output, via Theorem\cspace\ref{theorem V satisfies (CP) -> Fv(k) not homoheneous}, the non-$\aleph_1$-homogeneity of $F_\varw(\kappa)$, for any variety $\varw$ extending $\varv$, and all uncountable $\kappa$'s.

\smallskip\noindent
Since any free basis uniquely determines the corresponding free group, the only data we really need are $Y$, $\varv$ and a set $X$ playing the role of basis of $A$. Thus, in the rest of this section we will provide a supply of such triples.
As a result, we will be able to prove Theorem\cspace\ref{main_cor}.

\begin{definition}\label{definition torsion-free variety}
	A variety of groups $\varv$ is said to be:
	\begin{enumerate}[(1)]
		\item \emph{torsion-free} if and only if $x^n=e$ is not a law of $\varv$, for all $0<n<\omega$;
		\item \emph{non-nilpotent} if and only $\varv$ contains a group which is not nilpotent;
		\item \emph{locally nilpotent} if and only if every finitely generated $\varv$-group is nilpotent.
	\end{enumerate}
\end{definition}
\noindent
Definition\cspace\ref{definition torsion-free variety}\cspace\emph{(1)} is equivalent to saying that $\varab\subseteq\varv$ (cf. \cite{pope1984}), i.e. to $\varv$ having $\mathbb{Z}$ as the $\varv$-free group of rank $1$. As remarked in \cite{bele}, \mbox{examples of torsion-free varieties are:}
\begin{enumerate}[$\bullet$]
	\item[$\bullet$] the variety of all groups $\varg$;
	\item the abelian variety $\varab$;
	\item the variety of solvable groups of a given class;
	\item the variety of (poly-)nilpotent groups of a given length.
\end{enumerate}
\begin{strategy}\label{strategy}
	Our \emph{modus operandi} will be the following:
	\begin{enumerate}[(a)]
		\item consider a target variety $\varv$;
		\item fix an infinite set of generators $X=\{x_i\,:\,i<\omega\}$, which will act as basis of an absolutely free group $A$;
		\item select a set $Y=\{y_i\,:\,i<\omega\}$ whose elements are words in $X$ of the form
		\begin{equation*}
			y_i=x_{ki}w(x_{ki+1},\ldots,x_{ki+h}),\quad\text{or}\quad y_i=x_{ki}^{-1}w(x_{ki+1},\ldots,x_{ki+h}),
		\end{equation*}
		for a suitable group word $w(z_1,\ldots,z_h)\in L$, and some given values $k$ and $h$ $($depending on $\varv)$ s.t. $0< k,h<\omega$;
		\item define $B$ as the subgroup of $A$ generated by $Y$;
		\item transport the quotient $C=A/N_A(B)$ in $\varv$ by the corresponding verbal functor, and show that $C_\varv\ast_\varv F_\varv(\aleph_0)$ is not $\varv$-free.
	\end{enumerate}
\end{strategy}
Intuitively, \emph{(c)} says that the $y_i$'s are defined by an iterated coding in terms of the $x_i$'s. We can think to $h$ as an index of the complexity of such coding, and to $k$ as the pitch of the iteration.

\smallskip\noindent
For any choice of the word $w(z_1,\ldots,z_h)$ in \emph{(c)}, conditions \emph{(2)-(4)} of Theorem\cspace\ref{application theorem} are automatically satisfied. Indeed, by a classic result of Nielsen and Schreier (see\cspace\cite{lyndon-schupp}, Prop.\cspace2.11, pg.\cspace8), every subgroup of an absolutely free group is itself absolutely free. Furthermore, absolutely free groups of finite rank are known to be Hopfian (see\cspace\cite{lyndon-schupp}, Prop.\cspace3.5, pg.\cspace14). Where, a group $A$ is said to be \emph{Hopfian} if every $\alpha\in \text{\normalfont End}(A)$ which is onto is an automorphism of $A$. This property allows us to show the following.
\begin{lemma}\label{lemma H_n f.f. K}
	Under assumptions (b)-(c) of \ref{strategy}, $B_n=\langle\{ y_i\,:\,i\leq n\}\rangle_B\leq_* A$, for all $n<\omega$.
\end{lemma}
\noindent
Finally, this means also that $Y$ is a $\varv$-basis of $B$. We can see it directly from the proof of Lemma\cspace\ref{lemma H_n f.f. K}, or taking the statement above as black-box and applying Lemma\cspace\ref{lemma existence of V-free complementary factor} and Lemma\cspace\ref{lemma characterization of f.f.}. In both cases, we have that $Y$ is union of an increasing chain $(Y_n:n<\omega)$ s.t. $Y_n\vcentcolon=\{y_i:i\leq n\}$ is a $\varv$-basis of $B_n$, for all $n<\omega$. Since $B=\bigcup_{n<\omega}B_n$, every $b\in B$ belongs to some $B_n$. So, there exists a unique reduced group word $u(\bar{z})$ s.t. $B_n\models b=u(\bar{y})$, for some $\bar{y}\in Y_n^{|\bar{z}|}$. Clearly, the same must be true for every $B_m$ with $m\geq n$, as $Y_m$ is $\varv$-basis extending $Y_n$. Therefore, if there were a different reduced group word $u'(\bar{z}')$ s.t. $B\models b=u(\bar{y})\wedge b=u'(\bar{y}')$ for some $\bar{y}'\in Y^{|\bar{z}'|}$, there would be an $m<\omega$ s.t. $n\leq m$ and $B_m\models b=u(\bar{y})\wedge b=u'(\bar{y}')$ as well, since in $\bar{y}'$ only finitely many elements of $Y$ occur. However, this is absurd, so $Y$ is a $\varv$-basis of $B$, as desired.

\smallskip\noindent
For ease of reading, we consider a simple (but crucial) example before proceeding with the general proof of Lemma\cspace\ref{lemma H_n f.f. K}.
\begin{example}\label{example H_n f.f. K}
	Let $y_i=x_ix_{i+1}^{-2}$, for all $i<\omega$. Then, for any $n<\omega$, $B_n=\langle\{ y_i\,:\,i\leq n\}\rangle_B\leq_* A$. Notice that this is an instance of Lemma\cspace\ref{lemma H_n f.f. K} with $k,h=1$ and $w(z)=z^{-2}$.
\end{example}
\begin{proof}[Proof of\cspace\ref{example H_n f.f. K}]
	Clearly, given any $n<\omega$, $A_{n+1}\vcentcolon=\langle \{x_i:i\leq n+1\}\rangle_A$ is absolutely free, with free basis $\{x_i:i\leq n+1\}$. So, if $Z\vcentcolon=\{y_i:i\leq n\}\cup\{x_{n+1}\}$ is a generating set for $A_{n+1}$, then it is a free basis as well, since, by the Universal Property, the mapping $f:\{x_i:i\leq n+1\}\rightarrow Z$ s.t. $f(x_i)=y_i$ for all $i\leq n$, and $f(x_{n+1})=x_{n+1}$, extends to a surjective endomorphism of $A_{n+1}$, which is an automorphism by the Hopfianity of $A_{n+1}$. We show that this is the case.
	
	\smallskip\noindent
	From the relation $y_n=x_nx_{n+1}^{-2}$, we get $x_n=y_nx_{n+1}^2$, witnessing that $x_n\in\langle Z\rangle_{A_{n+1}}$. Hence, $x_{n-1}=y_{n-1}x_n^2=y_{n-1}(y_nx_{n+1}^2)^2$, and $x_{n-1}\in\langle Z\rangle_{A_{n+1}}$ as well. Going on in this way, it is easy to see that all $x_0,\ldots,x_n$ can be expressed as words in $Z$, as desired.
	
	\smallskip\noindent
	Since $A_{n+1}\leq_*A$, by Lemma\cspace\ref{lemma extension automorphisms of coproducts}, the automorphism $\phi'\in\mathrm{Aut}(A_{n+1})$ induced by the assignment $f:\{x_i:i\leq n+1\}\rightarrow Z$ naturally extends to an automorphism $\phi\in\mathrm{Aut}(A)$. Clearly, $\phi$ maps $\varv$-bases in $\varv$-bases. So, the inclusion $Z\subseteq\phi(X)$ witnesses that also $B_n\leq_* A$, by direction $[2\Rightarrow1]$ of Lemma\cspace\ref{lemma characterization of f.f.}.
\end{proof}
\noindent
The idea behind the argument above is that it does not matter the complexity of the coding word $w$ (or the pitch $k$) we choose: as long as the coding is performed as in \emph{(c)}, it is always possible to stop the iteration at any finite stage $n$ and find a minimal independent set $Z$ which extends $\{y_i:i\leq n\}$ and makes the following system solvable in $Z$:
\begin{equation}\tag{$*_{5}$}\label{eq lemma solutions of finite factorizations}
		\begin{cases}
			x_{kn}&=y_nw(x_{kn+1},\ldots,x_{kn+h})^{-1}\\
			x_{k(n-1)}&=y_{(n-1)}w(x_{k(n-1)+1},\ldots,x_{k(n-1)+h})^{-1}\\
			&\vdots\\
			x_0&=y_0w(x_{1},\ldots,x_{h})^{-1}.
		\end{cases}
\end{equation}
\begin{proof}[Proof of\cspace\ref{lemma H_n f.f. K}]
We consider the case of $y_i=x_{ki}w(x_{ki+1},\ldots,x_{ki+h})$, with $i<\omega$. The case of $y_i=x_{ki}^{-1}w(x_{ki+1},\ldots,x_{ki+h})$ follows, mutatis mutandis, by the same argument.

\smallskip\noindent
In light of Lemma\cspace\ref{lemma characterization of f.f.}, it suffices to show that $\{y_i\,:\,i\leq n\}$ extends to a basis of $A$. To this extent, we let:
\begin{equation*}
	Z\vcentcolon=\{y_i\,:\,i\leq n\}\cup\left(\,\bigcup_{i<n}\{x_{ki+1},\ldots,x_{k(i+1)-1}\}\right)\cup\{x_{kn+1},\ldots,x_{kn+h}\}.
\end{equation*}
We claim that $Z$ is a generating set for $A_{kn+h}\vcentcolon=\langle \{x_i\,:\,i\leq kn+h\}\rangle_A$. If this is the case, then the function $f:\{x_i:i\leq kn+h\}\rightarrow Z$ which maps $f(x_{ki})=y_i$, for all $i\leq n$, and which is the identity elsewhere, naturally extends to an endomorphism $\phi'$ of $A_{kn+h}$. Since $|Z|=kn+h+1$, $\phi'$ is surjective. So, it is an automorphism, because $A_{kn+h}$ is Hopfian. 

\smallskip\noindent
Now, as witnessed by $\phi'$, $Z$ itself is a free basis of $A_{kn+h}$. Hence, by direction $[2\Rightarrow 1]$ of Lemma\cspace\ref{lemma characterization of f.f.}, $B_n=\langle\{ y_i\,:\,i\leq n\}\rangle_B\leq_* A_{kn+h}$, as $\{y_i:i\leq n\}\subseteq Z$. 
Further, $A_{kn+h}$ is clearly a free factor of $A$, by definition. So, by Lemma\cspace\ref{lemma extension automorphisms of coproducts} applied to $A_{kn+h}\leq_*A$, there is some $\phi\in\mathrm{Aut}(A)$ s.t. $\phi'\subseteq\phi$. Since $\phi$ maps $\varv$-bases in $\varv$-bases, the inclusion $Z\subseteq\phi(X)$ shows that also $B_n=\langle\{ y_i\,:\,i\leq n\}\rangle_B\leq_* A$, again by direction $[2\Rightarrow1]$ of Lemma\cspace\ref{lemma characterization of f.f.}.

\smallskip\noindent
It only remains to see that $Z$ is a generating set of $A_{kn+h}$. The only elements of $\{x_i : i\leq kn+h\}$ which do not already belong to $Z$ are those of the form $x_{ki}$, for all $i\leq n$. We show that they admit expression as words in $Z$.

\smallskip\noindent
If $h<k$, it suffices to invert the defining relation of the $y_i$'s, obtaining the system \eqref{eq lemma solutions of finite factorizations}. Since $y_i,x_{ki+1},\ldots,x_{ki+h}\in Z$, for all $i\leq n$, we have the thesis.

\smallskip\noindent
If $k\leq h$, the first equation of \eqref{eq lemma solutions of finite factorizations} witnesses that $x_{kn}\in\langle Z\rangle_{A_{kn+h}}$. Hence, a substitution in the second equation of \eqref{eq lemma solutions of finite factorizations} of each occurrence of $x_{kn}$ with the term $y_nw(x_{kn+1},\ldots,x_{kn+h})^{-1}$ gives an expression of $x_{k(n-1)}$ as word in $Z$, so that $x_{k(n-1)}\in \langle Z\rangle_{A_{kn+h}}$ as well. Proceeding in this way, it is possible to solve any equation of system \eqref{eq lemma solutions of finite factorizations} by replacing all the occurrences of the $x_{k(i+1)}$'s with the terms in $Z$ we have obtained by the previous equations. In the end, all the $x_{ki}$'s will belong to $\langle Z \rangle_{A_{kn+h}}$, as desired.
\end{proof}
\noindent
The only delicate point of our strategy will be ensuring that $C_\varv\ast_\varv F_\varv(\aleph_0)$ is not $\varv$-free. If this is the case, Theorem\cspace\ref{application theorem} applies to a triple $(X,Y,\varv)$ as above.
\begin{1.2theorem}
	Let $V$ be a variety of groups satisfying at least one the following:
	\begin{enumerate}[(1)]
		\item $V$ is torsion-free;
		\item $V$ is not nilpotent but locally nilpotent;
		\item $V$ contains a finite non-nilponent group.
	\end{enumerate}
	Then, for every $\kappa \geq \aleph_1$ we have that $F_V(\kappa)$ is not $\aleph_1$-homogeneous. Furthermore, $F_V(\aleph_0)$ has an elementary subgroup which is not a $V$-free factor of $F_V(\aleph_0)$. 
\end{1.2theorem}
\begin{proof}[Proof of\cspace\ref{main_cor}]
	Let $A$ be the absolutely free group generated by a set $X=\{x_i\,:\,i<\omega\}$, and $B$ be the subgroup of $A$ freely generated by $Y=\{y_i:i<\omega\}$, where, for all $i<\omega$, we assume:
	\begin{enumerate}[$\bullet$]
		\item $y_i\vcentcolon=x_ix_{i+1}^{-2}$ in Case\,\emph{(1)}\,;
		\item $y_{i}\vcentcolon=x_{2i}^{-1}[x_{2i+1},x_{2i+2}]$ in Cases\cspace\emph{(2)} and\;\emph{(3)}.
	\end{enumerate}
	For each of \emph{(1)-(3)}, we claim that Theorem\cspace\ref{application theorem} applies to the tuple $(A,B,Y,\varv)$.
		
	\smallskip\noindent
	Notice that in each case the assumptions of Lemma\cspace\ref{lemma H_n f.f. K} are satisfied, so $B_n=\langle\{y_i:i\leq n\}\rangle_B\leq_* A$, for all $n<\omega$. In light of the previous discussion, the only condition we need to check is that the product $C_\varv\ast_\varv F_\varv(\aleph_0)$ is not $\varv$-free, with $C\vcentcolon= A/N_A(B)$.
	
	\smallskip\noindent
	The last two cases follow directly from an application of Mekler's results in\cspace\cite{mekler_groups} to our new framework: Lemma\cspace12 for $\varv$ not nilpotent but locally nilpotent (Case\,\emph{(2)}), and Lemma\cspace16 for $\varv$ containing a finite non-nilpotent group (Case\,\emph{(3)}).
	
	\smallskip\noindent
	For Case\,\emph{(1)}, i.e. for $\varv$ being torsion-free, necessarily $\varab\subseteq\varv$.
	So, in order to apply Theorem\cspace\ref{application theorem} to the pair $\varab\subseteq\varv$, it suffices to show that $C_\varab\oplus F_\varab(\aleph_0)$ is not free abelian.
	
	\smallskip\noindent
	In particular, we claim that $C_\varab$ is not contained in any free abelian group. To this extent, we denote by $Q$ the additive group of fractions $\big\langle 1/{2^i} \,:\, {i<\omega} \big\rangle_\mathbb{Q}$, and by $\hat{c}$ the coset of $c$ in $C_\varab$, for any $c\in C$. 
	
	\smallskip\noindent
	Observe that the quotient $C=A/N_A(B)$ has presentation:
	\begin{equation*}
		C=\langle x_i \,\vert\, x_i=x_{i+1}^2\,;\, i<\omega \rangle,
	\end{equation*}
	as we assumed $B\vcentcolon=\langle Y\rangle_A$, and $y_i\vcentcolon=x_ix_{i+1}^{-2}$ for any $i<\omega$. Hence, for all $i<\omega$, the relation $\hat{x}_i=\hat{x}_{i+1}^2$ holds in $C_\varab$. So, the assignment $\hat{x}_i\mapsto 1/2^i$ naturally extends to an isomorphism from $C_\varab$ onto $Q$, since the relation ${1}/{2^i}=2\cdot{1}/{2^{(i+1)}}$ holds in $Q$, for all $i<\omega$. However, $Q$ is $2$-divisible. Therefore, $C_\varab$ is not contained in any free abelian group, as desired.
\end{proof}

\end{document}